\documentclass[12pt]{amsart}
\usepackage{amssymb,latexsym,xspace,amsbsy,amsfonts,latexsym}
\usepackage{enumerate}

\makeatletter
\@namedef{subjclassname@2010}{%
  \textup{2010} Mathematics Subject Classification}
\makeatother

\swapnumbers
\theoremstyle{plain}

\numberwithin{equation}{section}

{\theoremstyle{definition}\newtheorem{num}{}[section]} 

\newcommand{\beq}{\begin{equation}}
\newcommand{\eeq}{\end{equation}}

\newtheorem{theorem}[num]{Theorem}
\newtheorem*{mainthm}{Main Theorem}
\newtheorem{lemma}[num]{Lemma}
\newtheorem{proposition}[num]{Proposition}
\newtheorem{corollary}[num]{Corollary}

{\theoremstyle{definition}}
{\theoremstyle{definition}}

{\theoremstyle{definition}}

{\theoremstyle{definition}\newtheorem{remark}[num]{Remark}}

\newcommand{\D}{\mathbb D}
\newcommand{\R}{\mathbb R}
\newcommand{\C}{\mathbb C}

\newcommand{\BH}{{\mathcal B}(H) }
\newcommand{\BX}{{\mathcal B}(X) }
\newcommand{\Ker}{\operatorname{Ker}}
\newcommand{\Ran}{\operatorname{Ran}}

\newcommand{\pocc}{\mathcal{S}}
\newcommand{\ep}{\varepsilon}
\newcommand{\ffi}{\varphi}
\newcommand{\tfi}{\widetilde{\ffi}}
\newcommand{\ld}{\lambda}
\newcommand{\sm}{\sigma}

\begin{document}


\baselineskip=17pt


\title[Averaged products of projections in Banach spaces]{Geometric, spectral
and
asymptotic properties
of averaged products of projections in Banach spaces}
\author{Catalin Badea }
\address{Laboratoire Paul Painlev\' e, Universit\'e Lille 1, CNRS UMR 8524, B\^at. M2,
F-59655 Villeneuve d'Ascq Cedex, France
}
\email{badea@math.univ-lille1.fr}
\author{Yuri I. Lyubich}
\address{Department of Mathematics, Technion, 32000, Haifa, Israel}
\email{lyubich@tx.technion.ac.il}

\date{}

\begin{abstract}
According to the von Neumann-Halperin and Lapidus theorems, in a Hilbert space
the iterates of products or, respectively, of convex combinations of
orthoprojections are strongly convergent.
We extend these results to the iterates of convex combinations of products
of some projections in a complex Banach space. The latter is assumed
uniformly convex or uniformly smooth for the orthoprojections,
or reflexive for more special projections,
in particular, for the hermitian ones. In all
cases the proof of convergence is based on a known criterion
in terms of the boundary spectrum.
\end{abstract}



\maketitle

\section{Introduction and background}

\begin{num}
{\bf What this paper is about.}
Let $H$ be a Hilbert space, and let
$M_1, \dots,M_N$ be closed subspaces of $H$. Denote by $P_k$
the orthoprojection onto $M_k$, and let
$T = P_1P_2 \cdots P_N$. It was proved by von Neumann \cite{neumann} for $N=2$ and by
Halperin \cite{halperin} for any $N$ that $T^n$ with $n\rightarrow\infty$
converges strongly to the orthoprojection onto $M_1\cap M_2 \cap \cdots \cap M_N$.
The same was proved by Lapidus \cite{lapidus} for
$T = \sum_{k=1}^N \alpha_kP_k$ with $\alpha_k > 0$, $1\le k\le N$,
and $\sum_{k=1}^N\alpha_k = 1$.
Some different proofs of these results were recently given in \cite{MR}.
The von Neumann-Halperin and Lapidus theorems were generalized
to uniformly convex Banach spaces by Bruck and Reich \cite{bruck/reich}
and Reich \cite{reich/LMA}, respectively. For a survey see \cite[Chapter 9]{deutsch:book}.

In the present paper we consider the situation when $T$
is a convex combination of products of some projections in a complex Banach space.
With some concordance between its geometry (uniform convexity or
uniform smoothness, or reflexivity) and a class of projections
(orthoprojections, hermitian projections, etc.)
we establish a spectral property of $T$ which
implies the strong convergence of $T^n$ as $n\rightarrow\infty$.
The necessary background is presented below.
\end{num}

\begin{num}
{\bf Spaces and operators.}
From now on we denote by $X$ a complex
Banach space and by $\BX$ the Banach algebra of linear bounded operators
on $X$. The identity operator will be denoted by $I$.

Recall that a space $X$ is said to be \emph{uniformly convex} if for every
$\ep\in (0,1)$ there exists $\delta\in (0,1)$ such that for any two vectors $x$ and $y$
with $\|x\|\le 1$ and $\|y\|\le 1$ the inequality $\|x+y\|/2>1-\delta$
implies $\|x-y\|<\ep$. Accordingly, the nondecreasing function
$$\delta_X(\ep)=\inf\left\{1-\frac{\|x+y\|}{2}:\|x\|\le 1,\|y\|\le 1,\|x-y\|\ge\ep\right\}$$
is called the \emph{modulus of convexity} of the space $X$. This classical definition, due to
Clarkson \cite{clarkson}, can be formally
applied to  all Banach spaces, so the uniformly convex spaces are just those which
satisfy $\delta_X(\ep)>0$ for all $\ep$.
Every Hilbert space $H$ is uniformly convex, its modulus of convexity is
$$\delta_H(\ep) = 1-\sqrt{1 -\frac{\ep^2}{4}}.$$
For more information on the modulus of convexity see e.g. \cite{BHW}, \cite{GR} and the
references therein.

A space $X$ is called \emph{uniformly smooth} if for every
$\ep > 0$ there exists $\delta > 0$ such that the inequality $\|x+y\|+\|x-y\|<2+\ep\|y\|$
holds for any two vectors $x$ and $y$ with $\|x\|= 1$ and $\|y\|\le\delta$.
A relevant modulus of smoothness was introduced by Day \cite{day}.
However, for the purposes of this paper we only need to know that
all uniformly convex and all uniformly smooth spaces are reflexive and a space $X$ is
uniformly smooth if and only if its dual $X^*$ is uniformly convex, see e.g. \cite{LT}.

Let $H$ be a Hilbert space. An operator $T\in\BH$ is hermitian
($\equiv$ self-adjoint) if and only if $\|\exp(itT)\| = 1$ for all real $t$. In any Banach space $X$
the latter property is a definition of a hermitian operator. (In \cite{lyu1}
such operators were called $\emph{conservative}$. This is just the case when
$T$ and $-T$ are \emph{dissipative}, i.e. generate semigroups of contractions
\cite{lumphi}).

Note that every real combination of pairwise commuting
hermitian operators is hermitian as well. In particular,
the operator $T-\alpha I$ is hermitian for any hermitian $T$ and any real $\alpha$.

For any operator $T\in\BX$ its spectrum is usually denoted by $\sigma(T)$.
If $T$ is hermitian then  $\sigma(T)\subset\R$.
If $T$ is a contraction, i.e. $\|T\|\le 1$, then $\sigma(T)\subset\overline{\D}$,
where $\D$ is the open unit disk in the complex  plane. The intersection of
$\sigma(T)$ with the unit circle $\partial\D$
is called the \emph{boundary spectrum} of the contraction $T$. Every
point $\ld\in\sigma(T)\cap \partial\D$ of the boundary spectrum is an
approximate eigenvalue, i.e. there is a sequence
of vectors $x_k$ of norm 1 such that $Tx_k-\ld x_k\rightarrow 0$.
The boundary spectrum may be empty. This happens if and only if there is
$n\ge 1$ such that $T^n$ is a strict contraction (i.e. $\|T^n\|<1$) or,
equivalently, $\|T^n\|\to 0$ as $n\to\infty$.
\end{num}
\begin{num}
{\bf Classes of contractions.}
A contraction is called \emph{primitive} if its boundary spectrum is
at most the singleton $\{1\}$. \emph{If the space $X$ is reflexive then the
iterates $T^n$ of any primitive contraction $T\in\BX$ are strongly convergent.}
This fact is the key to all convergence problems studied in the present paper.
Actually, it is a purely logical combination of two known general results:

1) If the space $X$ is reflexive then every contraction $T$ with at most countable
boundary spectrum is almost periodic, i.e. all orbits $(T^nx)_{n\ge 0}$
are precompact \cite{vulyu}.

2) In any Banach space the iterates of any
primitive  almost periodic contraction are strongly convergent \cite{jami}.
(See also \cite{lyu2} for a general theory of almost periodic operator
semigroups.)

An alternative proof (see Section 4 of the present paper) uses the
Katznelson-Tzafriri theorem \cite{KaTz}: \emph{in any Banach space}
$$\lim_{n\to\infty} \|T^{n}-T^{n+1}\| = 0$$
\emph{for every primitive contraction} $T$.

Note that all the results stated above for contractions are automatically
true for any power bounded operator $T\in\BX$ since $T$ is a contraction in an
equivalent norm on $X$.  On the other hand, even the weak convergence of $T^n$
implies the power boundedness of $T$.

The following geometric condition was introduced by Halperin in \cite{halperin}:
\begin{equation}
\label{H}
\tag{H} \textrm{ there is} \quad K \ge 0 \quad \textrm{such that }
\quad \|x-Tx\|^2 \le K\left( \|x\|^2 - \|Tx\|^2 \right) \quad (x\in X) .
\end{equation}
Under this condition (the same as ($K$) in \cite{DR}), $T$ is a contraction, and all strict
contractions
satisfy \eqref{H}. We denote by $K(T)$ the smallest value of $K$.
In particular, $K(I)=0$.

Halperin proved that in a Hilbert space the iterates of every
\eqref{H}-contraction are strongly convergent. In fact, this is
true in any reflexive Banach space. Indeed, from \eqref{H} it follows that
\begin{equation}
\label{S}
\tag{S} \|x_k\|\le 1,\quad \|Tx_k\| \to 1
\Rightarrow x_k-Tx_k \to 0  \quad \textrm{strongly}.
\end{equation}
However, \emph{every \eqref{S}-contraction is primitive}. Indeed, let
$\|Tx_k-\lambda x_k\|\to 0$ for a $\lambda\in\partial\D$ and
a sequence of normalized vectors $x_k$. Then $\|Tx_k\| \to 1$,
hence $\|Tx_k-x_k\|\to 0$ by condition \eqref{S}. Therefore,
$\lambda=1$. As a result, \emph{the iterates of every \eqref{S}-contraction
in a reflexive Banach space are strongly convergent.}

In Hilbert space this was proved in \cite{amemiya/ando}, where the condition
\eqref{S} appears together with its weak version
\begin{equation}
\label{W}
\tag{W} \|x_k\|\le 1,\quad \|Tx_k\| \to 1\Rightarrow
x_k-Tx_k \to 0  \quad \textrm{weakly},
\end{equation}
and the correspondig convergence result. The latter was extended to the
reflexive Banach space in \cite{DKR}.

Obviosly, the condition \eqref{W} implies
\begin{equation}
\label{W'}
\tag{W'} \|Tx\| = \|x\|\Rightarrow Tx = x .
\end{equation}
Conversely, \eqref{W'} implies \eqref{W} if the space is Hilbert
(see [3]) or, more generally, if it is a reflexive Banach space with a weakly
sequentially continuous duality map (see [12]).

Note that for the strict contractions the conditions \eqref{S} and
\eqref{W} are formally fulfilled but empty in content.

In \cite{dye:IEOT} Dye proved that in a Hilbert space the condition
\eqref{H} is equivalent to
\begin{equation}
\label{D}
\tag{D} \textrm{there is}\quad r\in (0,1):\|T-rI\| \le 1-r.
\end{equation}
Obviously, under the condition \eqref{D} the operator $T$ is a contraction.
Hence, $\|T-rI\| \ge 1-r$, so, finally, $\|T-rI\|= 1-r$.

\emph{Every \eqref{D}-contraction is primitive.} Indeed,
if $\ld \in \sigma(T)$, then $\ld - r \in \sigma(T-rI)$, so
$|\ld - r| \le \|T-rI\| \le 1-r $, whence $\ld =1$ for $|\ld|=1$.
Thus,\emph{ the iterates of every
\eqref{D}-contraction in a reflexive Banach space are strongly convergent.}
\end{num}
\begin{num}
{\bf Projections.}
Recall that a linear operator $P\in\BX$ is called a \emph{projection}
if $P^2=P$ or equivalently, $\Ker(P)=\Ran(I-P)$. Obviosly,
$\|P\|\ge 1$ if $P\neq 0$. A projection $P$ is called an \emph{orthoprojection}
if it is a contraction, i.e. $\|P\| = 1$ or $P=0$.
In Hilbert space this definition is equivalent to the standard one:
the subspaces $\Ker(P)$ and $\Ran(P)$ are mutually orhogonal. Equivalently,
this means that $P$ is hermitian. \emph{In any Banach space every hermitian
projection is an orthoprojection}. Indeed, for any projection $P$ we have
\begin{equation}\label{hop}
\exp(itP) = (I-P) + e^{it}P.
\end{equation}
Hence,
$$ P=\frac{1}{2\tau}\int_{-\tau}^{\tau}\exp(itP)e^{-it}\textrm{d}t$$
that yields $\|P\|\le 1$ if $P$ is hermitian.  However, if $P$ is a
hermitian projection then so is $I-P$ , while for the
orthoprojections this is not true, in general. Another specific feature
of the non-Hilbert situation is that for some subspaces the
orthoprojections do not exist. We refer the
reader to \cite{ortho} and \cite{berkson} for more details and references.

For our purposes it is important to note that \emph{all \eqref{D}-projections
are orthoprojections}.
Also note that \emph{every hermitian projection $P$ satisfies}
\eqref{D} with $r=1/2$, i.e. it is a\emph{$u$-projection} in the sense
of \cite{GKS}. This immediately follows from \eqref{hop}
by taking $t=\pi$. Obviously, if $P$ is a $u$-projection
then so is $I-P$ and both are orthoprojections.
\begin{mainthm}
\label{thm:main}
Let $P_1, \cdots, P_N$ be some orthoprojections in a complex Banach space $X$,
and let $\pocc = \pocc(P_1,\cdots,P_N)$ be the convex multiplicative semigroup
generated by $P_1, \cdots, P_N$, i.e. the convex hull
of the semigroup consisting  of all products with factors
from $\{P_1, \cdots, P_N\}$. Assume that one of the following conditions
is satisfied:
\begin{itemize}
\item[(i)] the space $X$ is uniformly convex;

\item[(ii)] the space $X$ is uniformly smooth;

\item[(iii)] the space $X$ is reflexive and all $P_k$ are of class \eqref{D}.
\end{itemize}
Then for every operator $T\in\pocc(P_1,\cdots,P_N)$ the iterates $T^n$ converge strongly
to an orthoprojection $T^{\infty}$.
In addition, if $P_k$ are of class \eqref{W'} then
\begin{equation}
\label{eq:ran}
\Ran(T^{\infty})= \cap_{k\in F_T}\Ran(P_k)
\end{equation}
where $F_T$ is the set of all indices $k$ occurring in the decomposition
of $T$ as a member of $\pocc(P_1,...,P_N)$.
The formula \eqref{eq:ran} is true in the class of all
orthoprojections if the space $X$ is
uniformly convex or uniformly smooth and strictly convex.
\end{mainthm}
Recall that a Banach space is called \emph{strictly convex}
if all points of its unit sphere are extreme.

In the case (i) the strong convergence of $T^n$, where $T$ is a product
or convex combination of orthoprojections, was proved in
\cite{bruck/reich} and in \cite{reich/LMA}, respectively. The space $X$
in these papers is real, but the results are automatically
true for the complex unifomly convex spaces by realification. On the other hand,
there is an example of divergence in $l^{\infty}_{4,\R}$, i.e in $\R^4$ endowed with the
max-norm (\cite{bruck/reich}, p.464).  Another related example
is in \cite{MR03}. In fact, there is an example even
in $l^{\infty}_{2,\R}$, a fortiori, in $l^{\infty}_{2,\C}$. Namely, let
$$
P_1 =\left(
\begin{array}{cc}
1 & 0\\
-1 & 0
\end{array}
\right) \quad \textrm{ and } P_2 = \left(
\begin{array}{cc}
0 & 1\\
0 & 1
\end{array}
\right).
$$
Then
$$(P_1 P_2)^n =\left(
\begin{array}{cc}
0 & (-1)^{n+1}\\
0 & (-1)^{n}
\end{array}
\right),$$
so the iterates  $(P_1P_2)^n$ are divergent.

The space in our example
is not strictly convex. An open question is about existence
of an example of divergence in a strictly convex space. For the affirmative
answer the space must be infinite-dimensional since every finite-dimensional
strictly convex space is uniformly convex.

For any Banach space $X$ and its closed subspace $M$, we denote by $P_M(x)$,
$x\in X$, the set of points in $M$ whose distance to $x$ is minimal.
If $X$ is reflexive then the set $P_M(x)$ is not empty for every $x$.
If, in addition, $X$ is strictly convex, then $P_M(x)$ is a singleton.
In this situation $P_M(x)$ can be considered as a point in $X$ and $P_M$
as a mapping $X\to X$, a \emph{nearest point projection} onto $M$.
In general, this 'projection'  is nonlinear. However, in a Hilbert space
$P_M$ coincides with the orthoprojection onto $M$.

For a strictly convex reflexive space $X$ with $\dim X>2$ Stiles proved
in \cite{stiles} that if $(P_MP_N)^n$
converges strongly to $P_{M\cap N}$ for every pair $(M,N)$ of closed subspaces
of $X$, then $X$ is a Hilbert space. Thus, the von Neumann theorem cannot be
extended to the nearest point projections in a non-Hilbert space.
See however \cite[Lemma 3.1]{reich/LMA} for a relation between linear
nearest point projections and orthoprojections. This makes it
possible to obtain a counterpart of the Main Theorem for
linear nearest point projections. This observation was kindly
communicated to us by S.~Reich.

Note that the weak convergence of the iterates of a product or a convex combination
of orthoprojections in a uniformly smooth space follows from \cite{bruck/reich} and
\cite{reich/LMA} by duality.
\end{num}

\begin{num}
{\bf Organization of the paper.}
The next section contains some information on the Apostol modulus
$\ffi_T(\ep)$ and its modification $\tfi_T(\ep)$ for a contraction $T$
in a Banach space. In Section 3 we apply it to prove that
the classes \eqref{H}, \eqref{S}
and \eqref{D} are multiplicative semigroups, furthermore, \eqref{S} and \eqref{D}
are convex . This is an important ingredient of the proof of
the Main Theorem. The latter is given in Section 4 after a proof of the convergence of
the iterates of a primitive contraction in a reflexive Banach space.
We conclude with an Appendix (Section 5) where we study some relations between
the Apostol moduli and a geometric characteristic of the boundary
spectrum. This yields a new look at a generalization of the
Katznelson-Tzafriri theorem obtained by Allan and Ransford \cite{allan/ransford}.
\end{num}
\begin{num}
{\bf Acknowledgments.}
The first author was supported in part by the ANR-Projet Blanc DYNOP.
The work of the second author was carried out within the framework of the TODEQ program.
We thank Eva Kopeck\'a, Simeon Reich and the referee for helpful discussions,
remarks and references.
\end{num}

\section{The Apostol modulus}

\begin{num}
{\bf Definitions and basic facts.}
For a contraction $T \in \BX$ we consider the \emph{Apostol modulus}
$$\varphi_T (\varepsilon) = \sup \{  \|x - Tx\| :
\|x\| \le 1,\quad \|x\| - \|Tx\| \le \varepsilon\},\quad 0<\varepsilon \le 1 . $$
This function was introduced and studied by Apostol \cite{apostol}
in the case of Hilbert space.
For our purposes the following modification is convenient:
$$\widetilde{\varphi}_T (\varepsilon) = \sup \{\|x - Tx\| :
\|x\| \le 1,\quad 1 - \|Tx\| \le \varepsilon\}. $$
This definition is correct if and only if $\|T\|=1$ since this is the
only case when the set $\{x:\|x\| \le 1, 1 - \|Tx\| \le \varepsilon\}$
is nonempty for all $\varepsilon$. Thus, we will assume $\|T\|=1$ anytime
when dealing with $\widetilde{\varphi}_T (\varepsilon)$. On the other
hand, in all further
applications the case $\|T\|<1$ is trivial.

Obviously, both  functions $\ffi_T(\varepsilon)$ and $\tfi_T(\varepsilon)$ are
nondecreasing and
\begin{equation}
\label{eq:upp}
0\le\tfi_T(\varepsilon) \le \ffi_T(\varepsilon)\le \|I-T\|\le 2.
\end{equation}
Actually, the most interesting information relates to their behavior as $\ep\to 0$.
Accordingly,  we consider
$$ \varphi_T^0 = \lim_{\varepsilon \to 0} \varphi_T(\varepsilon)=
\inf_{\varepsilon > 0}\varphi_T(\varepsilon)\geq 0$$
and
$$\tfi_T^0 =\lim_{\varepsilon \to 0} \tfi_T (\varepsilon)=
\inf_{\varepsilon > 0}\tfi_T (\varepsilon)\geq 0.$$
It turns out that these limit values coincide. In this sense the difference
between the two versions of the Apostol modulus is not essential.
\begin{lemma}
\label{lem:basin}
If $T$ a contraction of norm 1 and $T\neq I$ then $\ffi_T(\varepsilon)>0$ for
all $\ep$ and
$$0\le\ffi_T(\varepsilon)\le
\tfi_T\left(\frac{\|I-T\|\varepsilon}{\ffi_T(\varepsilon)}+0\right).$$
\end{lemma}
\begin{proof}
Assuming $\ffi_T(\varepsilon)=0$ for an $\ep$, we obtain $\|x - Tx\|=0$
for all $x$ with $\|x\|\le\ep$, so $T=I$. Now let $T\neq I$.
Take $q\in (0,1)$ and find a vector $x$ such that
$$\|x\| \le 1,\quad \|x\| - \|Tx\| \le \varepsilon,
\quad \|x - Tx\|=q\theta$$
where $\theta=\ffi_T(\varepsilon)>0$. Then for the normalized vector
$z=x/ \|x\|$ we have
$$1-\|Tz\|\le\frac{\varepsilon}{\|x\|},\quad\|z - Tz\|=\frac{q\theta}{\|x\|}\ge q\theta,$$
whence
$$\tfi_T\left(\frac{\varepsilon)}{\|x\|}\right)\ge q\theta.$$
On the other hand,
$$\tfi_T\left(\frac{\varepsilon}{\|x\|}\right)\le
\tfi_T\left(\frac{\|I-T\|\varepsilon}{q\theta}\right)$$
since $\|x\|\ge q\theta/\|I-T\|.$ Thus,
$$q\theta\le \tfi_T\left(\frac{\|I-T\|\varepsilon}{q\theta}\right).$$
It remains to subsitute $\theta$ by $\ffi_T(\varepsilon)$ and pass
to the limit as $q\to 1$ .
\end{proof}
\begin{corollary}
\label{cor:equ}
$\tfi_T^0=\ffi_T^0$ for all contractions $T$ of norm 1.
\end{corollary}
\begin{proof}
Since $\tfi_I^0=\ffi_I^0=0$, one can assume $T\neq I$ and apply Lemma \ref{lem:basin}.
As $\ep\to 0$ we get $\ffi_T^0\le\tfi_T^0$. The opposite inequality is trival.
\end{proof}
From now on we denote by $\omega_T$ the common value of $\ffi_T^0$ and $\tfi_T^0$.
For instance, $\omega_I = 0$. Accordingly, \eqref{eq:upp} can be extended to
\begin{equation}
\label{eq:uppe}
0\le\omega_T\le\tfi_T(\varepsilon) \le \ffi_T(\varepsilon)\le\|I-T\|\le 2.
\end{equation}
\end{num}
\begin{theorem}
\label{thm:S}
$\omega_T= 0$ if and only if $T$ is of class \eqref{S}.
\end{theorem}
\begin{proof}
"If". There is a sequence of vectors $x_k$ such that
$\|x_k\|\le 1$, $1-\|Tx_k\|\le 1/k$ and
$\tfi_T(1/k)<2\|x_k - Tx_k\|$. The latter norm tends to zero
if $T$ satisfies conditon \eqref{S}.

"Only if". Let $\|x_k\|\le 1$ and $\|Tx_k\| \to 1$. Without loss
of generality one can assume $\|Tx_k\|<1$, otherwise, we change $x_k$ to
$q_kx_k$ where all $q_k\in (0,1)$ and $q_k\to 1$ as $k\to\infty$.
Since $\omega_T = 0$ we have
$\tfi_T(1-\|Tx_k\|)\to 0$, whence $\|x_k - Tx_k\|\to 0$ by the
the obvious inequality
$$\|x\| - \|Tx\| \le \tfi_T(1-\|Tx\|)\quad (\|x\|\le 1,\quad \|Tx\|<1).$$
\end{proof}
\begin{remark}
Theorem \ref{thm:S} remains in force for $\|T\|<1$ if we set $\omega_T=0$
in this case. The latter definition is natural. Indeed, if $\|T\|<1$ then
$$\ffi_T(\ep)\le \frac{\|I-T\|\ep}{1-\|T\|},$$
whence $\ffi_T^0=0$. (Recall that $\tfi_T^0$ is not defined for $\|T\|<1$.)
\end{remark}
\begin{remark}
Let $T$ be an isometry. Then $\ffi_T(\ep)=\|I-T\|$ for all $\ep$, hence
$\omega_T= \|I-T\|$, therefore, $\omega_T>0$ if $T\neq I$.
\end{remark}
\begin{num}
{\bf The Apostol modulus for orthoprojections.}
\label{num:214}
If $P$ is an orthoprojection, so that $\|P\|\le 1$, then
\begin{equation}
\label{eq:opo}
\|Px\|=\frac{1}{2}\|P(x+Px)\|\le \frac{1}{2}\|x+Px\|\le\|x\|
\end{equation}
Now let $\|x\| \le 1$, and let $1-\|Px\| \le \ep$. Then $\|Px\|\le 1$ and
$\frac{1}{2}\|x+Px\|\ge 1-\ep$. Hence, $\|x-Px\| \le \beta_X(\ep)$ where
$$\beta_X(\ep)=\sup\{\|x-y\|:\|x\|\le 1, \|y\|\le 1,\frac{\|x+y\|}{2}\ge 1-\ep\} .$$
This results in the inequality
\begin{equation}
\label{eq:amo}
\tfi_P(\ep)\le\beta_X(\ep).
\end{equation}

The function $\beta_X$ was introduced and investigated in \cite{BHW}.
It is closely related to the modulus of convexity. In particular,
$\lim_{\ep\to 0}\beta_X(\ep)=0$ if the space $X$ is uniformly convex,
otherwise, this limit is the supremum of those $\ep$ for which $\delta_X(\ep)=0$.
The latter quantity (or 0 if $X$ is uniformly convex) is called the
\emph{characteristic of convexity} of the space $X$, see \cite{GR}.
\end{num}
\begin{proposition}
\label{prop:usom}
If $P$ is an orthoprojection in a uniformly convex space then $\omega_P=0$.
\end{proposition}
\begin{proof}
This follows from the inequality \eqref{eq:amo} by passing to the limit as $\ep\to 0$.
\end{proof}
\begin{corollary}
\label{cor:us}
Every orthoprojection in a uniformly convex space is of class \eqref{S}.
\end{corollary}
\begin{remark}
This corollary can be obtained directly from \eqref{eq:opo}.
In this way Proposition \ref{prop:usom} follows from Theorem \ref{thm:S}.
\end{remark}

The uniform convexity of $X$ is not necessary for the existence of
\eqref{S}-orthoprojections. For instance, if a projection $P$
in $X$ is such that $\|x\| = \|Px\| + \|x-Px\|$
for all $x\in X$ (an $L$-\emph{projection} \cite{HWW})
then $P$ is an orthoprojection and $\omega_P=0$. Indeed, either
$P=I$ or $\ffi_P(\ep)=\ep$ for all $\ep$.
In this situation $X$ may not be uniformly convex. An example is $X=\ell^1$
where any coordinate projection is an $L$-projection.
\begin{remark}
\label{rem:strco}
From \eqref{eq:opo} it follows that \emph{every orthoprojection in a strictly
convex space is of class \eqref{W'}}.
\end{remark}

\section{Structure properties of classes \eqref{H}, \eqref{S} and \eqref{D}}

In this section we prove the following theorem.
\begin{theorem}
\label{thm:HSD}
In any Banach space the sets of contractions of classes \eqref{H}, \eqref{S} and \eqref{D}
are multiplicative semigroups. In addition, they are convex in the cases \eqref{S} and
\eqref{D}.
\end{theorem}
This theorem is an immediate consequence of the lemmas proven below.
\begin{lemma}
\label{lem:Halperin}
Let $A$ and $B$ be two contractions satisfying condition \eqref{H}.
Then the product $AB$ also satisfies \eqref{H} and
$$K(AB)\le 2\max (K(A),K(B)).$$
\end{lemma}
\begin{proof}
We have
$$\|x - ABx\|^2 \le (\|x - Bx\| + \|Bx - ABx\|)^2
  \le 2(\|x - Bx\|^2 + \|Bx - ABx\|^2),$$
whence

\begin{align*}
\|x - ABx\|^2 &\le 2K(B)(\|x\|^2 - \|Bx\|^2) + 2K(A)(\|Bx\|^2 - \|ABx\|^2)\\
 &\le 2\max (K(A),K(B))(\|x\|^2 - \|ABx\|^2) .
\end{align*}
\end{proof}
Thus, {\em the set of \eqref{H}-contractions is a multiplicative semigroup.}
\begin{remark}
If $T$ is an \eqref{H}-contraction then
$$\ffi_T(\ep) \le \sqrt{2K(T)\ep}.$$
Indeed, if $\|x\| \le 1$ and $\|x\| -\|Tx\| \le \ep$, then
$$\|x-Tx\|^2 \le K(T)\left(\|x\|^2 - \|Tx\|^2\right)
\le 2K(T)(\|x\| -\|Tx\|) \le 2K(T)\ep.$$

In particular, if $P$ is an orthoprojection in a Hilbert space $H$
then $$\|x-Px\|^2 = \|x\|^2-\|Px\|^2.$$
Thus, $P$ satisfies \eqref{H} with constant $K(P)=1$. Hence, $\ffi_P(\ep) \le \sqrt{2\ep}$.
\end{remark}
\begin{lemma}
\label{lem:22}
\begin{itemize}
\item[(i)] Let $A$ and $B$ be some contractions of norm 1. Then either $\|AB\|<1$ or
$$ \tfi_{AB}(\varepsilon) \le
\tfi_{A}(\tfi_{B}(\varepsilon)+ \varepsilon) + \tfi_{B}(\varepsilon) .$$
\item[(ii)] Let
$$T= \sum_{k=1}^N \alpha_kA_k$$
be a convex combination of contractions $A_k$ of norm 1, and let all $\alpha_k > 0$.
Then either $\|T\|<1$ or
$$ \tfi_{T}(\varepsilon) \le \sum_{k=1}^N\alpha_k\tfi_{A_k}(\alpha_k^{-1}\ep).$$
\end{itemize}
\end{lemma}
\begin{proof}
(i). Let $\|AB\|=1$. Then $\|A\|=\|B\|=1$, so the functions $\tfi_{A}$, $\tfi_{B}$
are well defined along with  $\tfi_{AB}$. Take any  vector $x$ such that
$\|x\| \le 1$ and $1-\|ABx\| \le \varepsilon$. Then
$$ \|x-ABx\| \le \|x-Ax\|+\|Ax-ABx\| \le \|x-Ax\|+\|x-Bx\|.$$
Thus,
$$\|x-ABx\|\le \tfi_A (1 - \|Ax\|) + \tfi_B (1 - \|Bx\|).$$
Let us estimate $1 - \|Ax\|$ and $1 - \|Bx\|$. We have
$$ 1 - \|Bx\| \le 1 - \|ABx\| \le \varepsilon $$
and then
 $$ 1 - \|Ax\| \le 1 + \|A(x - Bx)\| - \|ABx\|\le \|x - Bx\| + (1 - \|ABx\|).$$
Thus,
 $$ 1 - \|Ax\| \le\tfi_{B}(\varepsilon)+ \varepsilon.$$
As a result,
$$\|x-ABx\|\le \tfi_A (\tfi_{B}(\varepsilon)+ \varepsilon) + \tfi_B (\varepsilon).$$

(ii). Let $\|T\|=1$. Then all $\|A_k\|=1$, so the functions $\tfi_{A_k}$
are well defined along with  $\tfi_{T}.$  Take $x$ such that
$\|x\| \le 1$, $1-\|Tx\| \le \varepsilon$, i.e.
$$1 - \|\sum_{k=1}^N\alpha_kA_kx\|\le \varepsilon.$$
A fortiori,
$$\sum_{k=1}^N\alpha_k( 1-\|A_kx\|)\le \varepsilon,$$
whence $1 - \|A_kx\| \le \alpha_k^{-1}\ep$ for every $k$.
Hence,
$$\|x - Tx\|\le\sum_{k=1}^N \alpha_k\|x-A_kx\|
\le\sum_k\alpha_k \tfi(\alpha_k^{-1}\ep).$$
\end{proof}
As a consequence, if $\omega_A = \omega_B = 0$ then $\omega_{AB} = 0$,
and if all $\omega_{A_k} =0$ then $\omega_{T} = 0$. By Theorem \ref{thm:S}
{\em the set of \eqref{S}-contractions is a convex multiplicative semigroup}.

Now for a contraction $T$ we consider the set
$$R(T)=\{r\in (0,1):\|T - rI\| \le 1-r\}.$$
By definition, $T$ is a \eqref{D}-contraction if and only if $R(T)\neq\emptyset$.
\begin{lemma}
\label{lem:D}
For any contractions $A$ and $B$ if $r\in R(A)$ and $s\in R(B)$ then $rs\in R(AB)$
and $\alpha r + \beta s\in R(\alpha A + \beta B)$ with $\alpha>0$, $\beta>0$ and
$\alpha+\beta = 1$.
\end{lemma}
\begin{proof}
First, we have
\begin{align*}
\|AB - rsI\| &= \|A(B-sI) + s(A - rI)\|\\
 &\le \|B -sI\| + s\|A -rI\|\le 1-rs.
\end{align*}
Secondly,
\begin{align*}
\|(\alpha A + \beta B) - (\alpha r + \beta s)I\|&\le \alpha \|A -rI\| +
\beta \|B -sI\|\\
&\le \alpha (1-r) + \beta (1-s)=1-(\alpha r + \beta s). 
\end{align*}
\end{proof}
Thus, {\em the set of \eqref{D}-contractions is a convex multiplicative semigroup.}
The proof of Theorem \ref{thm:HSD} is complete.

\section{Proof of the Main Theorem}

The following general result is a key lemma in the proof of our Main Theorem.
\begin{theorem}
\label{thm:conv}
If $X$ is a reflexive space and $T$ is a primitive contraction in $X$ then the
iterates $T^n$ converge strongly. The limit operator $T^{\infty}$ coincides with
the orthoprojection $E_T$ onto the subspace $L=\Ker(I-T)$ along the closure
$M=\overline{\Ran(I-T)}$. The convergence is uniform if and only if $\Ran(I-T)$
is closed.
\end{theorem}
\begin{proof}
According to the classical ergodic theorem \cite{lorch}, the Ces\`aro means
of $(T^n)_{n\ge 0}$ converge strongly to the projection $E_T$ onto $L$ along $M$.
A part of this statement is that $X$ is the direct sum $L\oplus M$.
Let $x=u+v$ where $u\in L$, i.e. $Tu=u$, and $v\in M$, i.e.
$v=\lim_{k\to\infty}(z_k-Tz_k)$ for a sequence $(z_k)_{k\ge 0}$.
Given $\ep>0$, we take and fix $k$ such that $\|v- (z_k-Tz_k)\|<\ep$.
Then $\|T^nv- (T^n- T^{n+1})z_k\|<\ep$ for all $n$. Hence,
$\|T^nv\|<\ep+ \|T^n- T^{n+1}\|\|z_k\|<2\ep$
for large $n$ by the Katznelson-Tzafriri theorem \cite{KaTz}. Thus,
$\lim_{n\to\infty}T^nv=0$. As a result, $\lim_{n\to\infty}T^nx=u=E_Tx$,
i.e. $T^{\infty}=E_T$. The latter is an orthoprojection since $T$ is
a contraction.

Now suppose that $\Ran(I-T)$ is closed, i.e. $M=\Ran(I-T)$.
The operator $I-T$ acts bijectively on the invariant subspace M.
Since $M$ is closed, the inverse operator $S=((I-T)|M)^{-1}$
is bounded. Since $(T|M)^n= (T^n- T^{n+1})S$, we obtain $\|(T|M)^n\|\to 0$
by the Katznelson-Tzafriri theorem again. Conversely, if $T^n$ converges uniformly
then the same is true for the Ces\`aro means, and then $\Ran(I-T)$ is closed
(\cite{Lin}).
\end{proof}
An alternative proof is merely a logical combination of two results proved in \cite{vulyu} and
\cite{jami} as we indicated in the Introduction.

\begin{num}
\begin{proof}[Proof of the Main Theorem]
Let $T\in\pocc(P_1,\cdots, P_N)$
where $P_1, \cdots, P_N$ are some orthoprojections in a Banach space $X$.
Obviously, $T$ is a contraction. By Theorem \ref{thm:conv} it suffices to
show that $T$ is primitive in all cases (i)-(iii). Recall that all
contractions of classes \eqref{S} and \eqref{D} are primitive. (See Section 1.)

(i). The space $X$ is uniformly convex. Then by Corollary \ref{cor:us}
all $P_k$ are of class \eqref{S}. By Theorem \ref{thm:HSD} so is $T$.
Therefore, $T$ is primitive.

(ii). The space $X$ is uniformly smooth. Then $X^*$ is uniformly convex
and $T^* \in \pocc(P_1^*,\cdots, P_N^*)$.
All $P_k^*$ are orthoprojections since $\|A^*\|=\|A\|$ for any operator $A$.
Therefore, $T^*$ is primitive like $T$ in (i). Then $T$ is also primitive
since  $\sigma(A) =\sigma(A^*)$ for any operator $A$ and $T=T^{**}$ by
reflexivity  of $X$.

(iii). The space $X$ is reflexive. Since all $P_k$ are of class \eqref{D},
such is also $T$ by Theorem \ref{thm:HSD}. Thus, $T$ is primitive again.


To complete the proof of the Main Theorem we note that
the subspace $\Ran(T^{\infty})$ coincides with the subspace $\Ker(I-T)$
of fixed points of the operator $T$. Thus, it suffices to refer to
the following lemma and Remark \ref{rem:strco}.
\end{proof}
\begin{lemma}
\label{lem:kernels}
(i) Let $A$ and $B$ be some \eqref{W'}-contractions. Then
$$\Ker(I-AB) = \Ker(I-A)\cap \Ker(I-B).$$
(ii) Let $T$ be a convex combination of $N$ \eqref{W'}-contractions:
$T = \sum_{k=1}^N\alpha_kA_k$ with all $\alpha_k > 0$. Then
$$\Ker(I-T) = \cap_k\Ker(I-A_k).$$
\end{lemma}
\begin{proof}
In both cases the inclusion of the right-hand side into the left-hand side
is trivial. The proofs of the converse inclusions are as follows.

(i) For $x \in \Ker(I-AB)$ we have
$$\|x\| = \|ABx\| \le \|Bx\| \le \|x\|.$$
Therefore, $\|Bx\| = \|x\|$, whence $Bx=x$ and then $Ax=x$ by condition \eqref{W'}.

(ii) For $x \in \Ker(I-T)$ we have
$$\|x\| \le \sum_{k=1}^N \alpha_k\|A_kx\|
\le \sum_{k=1}^N \alpha_k\|x\| = \|x\|.$$
Thus, $\|A_kx\| = \|x\|$, hence $A_kx=x$ for every $k$.
\end{proof}
\begin{remark}
The same argument shows that {\em the contractions of class \eqref{W'} constitute
a convex multiplicative semigroup.}
\end{remark}
\end{num}

\section{Appendix: the amplitude of the boundary spectrum}

Let $T$ be a contraction in a Banach space $X$, and let the boundary spectrum
of $T$ be nonempty. We call the quantity
$$a_T=\max\{|\ld - 1|:\ld \in \sm(T), |\ld| =1\}$$
the \emph{amplitude of the boundary spectrum} of $T$. Obviously,
$0\le a_T\le 2$, and $a_T=0$ if and only if the contraction $T$ is primitive.
In view of Theorem \ref{thm:S}, the fact of the primitivity of the
\eqref{S}-contractions is a particular case of the following inequality.
\begin{proposition}
\label{prop:amp}
$a_T\le\omega_T$.
\end{proposition}
\begin{proof}
Let $\ld\in\sm(T)$, $|\ld| =1$. Then for every $\ep>0$ there exists a vector
$x$ of norm 1 such that $\|Tx-\ld x\|\le\ep$. Hence, $1-\|Tx\|\le\ep$ and
$$|\ld - 1|\le \|x-Tx\|+\|Tx-\ld x\| \le\tfi_T(\ep)+\ep.$$
The result follows as $\ep\to 0$.
\end{proof}
\begin{corollary}
\label{cor:amax}
If $a_T=2$ then $\omega_T=2$ and $\tfi_T(\ep)=\ffi_T(\ep)=2$ for all $\ep$.
Also, $\|I-T\|=2$ in this case.
\end{corollary}
\begin{proof}
We have $\omega_T\ge 2$. Now everything follows from \eqref{eq:uppe}.
\end{proof}
Obviously, $a_T=2$ if and only if $-1\in\sigma(T)$. Therefore, \emph{if $-1\in\sigma(T)$
then $\omega_T=2$}.
\begin{proposition}
If the space $X$ is uniformly convex and $\omega_T = 2$ then $a_T=2$.
\end{proposition}
\begin{proof}
We have $\tfi_T(\ep) = 2$ for every $\ep\in (0,1)$. By definition,
there is a vector $x=x(\ep)$ of norm 1 such that $\|x-Tx\|\ge 2-2\ep$.
Hence, $\|x+Tx\|\le\beta_X(\ep)$ where $\beta_X$ is the function defined
in Section 2. Since the space $X$ is uniformly convex, we have
$\lim_{\ep\to 0}\beta_X(\ep)=0$. A fortiori, $\lim_{\ep\to 0}\|x(\ep)+Tx(\ep)\|=0$.
This means that $-1\in\sigma(T)$.
\end{proof}
The amplitude $a_T$ is the maximal deviation of the boundary spectrum of $T$
from the point 1 in the metric of the complex plane. Alternatively,
one can use the metric of the unit circle. This "intrinsic" amplitude is
$$\tau_T=2\arcsin\left(\frac{a_T}{2}\right).$$
In \cite{allan/ransford} Allan and Ransford obtained
the following quantitative version of the Katznelson-Tzafriri theorem:
$$ \limsup_{n\to\infty} \|T^{n} - T^{n+1}\|\le
2\tan\left(\frac{\tau_T}{2}\right),\quad\tau_T<\pi.$$
In terms of the amplitude $a_T$ this means that
$$\limsup_{n\to \infty}\|T^{n} - T^{n+1}\|\le\frac{2a_T}{\sqrt{4-a_T^2}},\quad a_T<2.$$
Combining this result with Proposition \ref{prop:amp} we obtain
\begin{theorem}
\label{thm:Ssp1}
Let $T$ be a contraction acting on the complex Banach space $X$. If  $\omega_T < 2$ then
$$\limsup_{n\to \infty} \|T^{n} - T^{n+1}\| \le \frac{2\omega_T}{\sqrt{4-\omega_T^2}}.$$
\end{theorem}

\bibliographystyle{plain}
\bibliography{SM6746_Badea_Lyubich}

\end{document}